\documentclass[11pt]{amsart}
\usepackage{graphicx}
\newtheorem{thm}{Theorem}[section]
\newtheorem{cor}[thm]{Corollary}
\newtheorem{lem}[thm]{Lemma}

\theoremstyle{definition}

\theoremstyle{remark}

\numberwithin{equation}{section}
\begin{document}

\title[Solutions-Stability  of  Wilson's functional equation]{Solutions and stability of variant of  Wilson's functional equation}%
\author[ E. Elqorachi and A. Redouani]{ Elqorachi Elhoucien and  Redouani Ahmed}%
%\address{}%
%\email{}%

%\subjclass{}%

%\date{}%
%\dedicatory{}%
%\commby{}%
% ----------------------------------------------------------------
\begin{abstract}
In this paper we will investigate the solutions and stability of the
generalized variant of Wilson's functional equation
$$ (E):\;\;\;\;
f(xy)+\chi(y)f(\sigma(y)x)=2f(x)g(y),\; x,y\in G,$$ where $G$ is a
group, $\sigma$ is an involutive morphism of $G$  and $\chi$ is a
character of $G$. (a) We solve $(E)$ when $\sigma$ is an involutive
automorphism, and we obtain  some properties  about solutions of
$(E)$ when $\sigma$ is an involutive anti-automorphism. (b) We
obtain the Hyers Ulam stability of equation $(E)$.  As an
application, we prove the superstability of the functional equation
$f(xy)+\chi(y)f(\sigma(y)x)=2f(x)f(y),\; x,y\in G.$
\end{abstract}
\maketitle
% ----------------------------------------------------------------
\section{Introduction} D'Alembert's functional equation
\begin{equation}\label{eq11}
    f(x+y)+f(x-y)=2f(x)f(y), x,y\in G
\end{equation} also called the cosine functional equation has a long
history going back to J.d'Alembert. Equation (\ref{eq11}) plays an
important role in determining the sum of two
vectors in various Euclidean and non-Euclidean geometries.\\
The continuous solutions $f$: $\mathbb{R}\longrightarrow \mathbb{C}$
of d'Alembert's functional equation (\ref{eq11}) are known: A part
from the trivial solution $f=0$, they are
$f_{\lambda}(x)=\cos(\lambda x)$, $x\in \mathbb{R}$ where the
parameter $\lambda$ ranges over $\mathbb{C}$ (see for example
\cite{ac1})
\\Several authors have determined the general solution $f$:
$G\longrightarrow \mathbb{C}$ of the following generalization of
d'Alembert's functional equation
\begin{equation}\label{eq12}
    f(xy)+f(x\sigma(y))=2f(x)f(y),\; x,y\in G
\end{equation}in abelian case and in non abelian case. \\Probably the
first result in non abelian group was obtained by Kannappan
\cite{ka1}. Under the condition $f(xyz)=f(yxz)$ for all $x,y,z\in
G$, the solutions  of equation (\ref{eq12}) are of the form
$f(x)=\frac{\phi(x)+\phi(\sigma(x))}{2}$, where $\phi$ is multiplicative.\\
There has been quite a development of the theory of d'Alembert's
functional equation (\ref{eq11}) during the last years, on non
abelian groups, as shown in works by Dilian yang about compact
groups \cite{di1,di2,di3}, Stetk\ae r \cite{st1} for step 2
nilpotent groups, Friis \cite{fri1} for results on Lie groups and
Davison \cite{da1,da2} for general groups, even monoids. The most comprehensive recent study is by stetk\ae r  \cite{st10,st11}. \\
Recently, Stetk\ae r \cite{st2} obtained the complex valued
solutions of the following version of d'Alembert's functional
equation
\begin{equation}\label{eq13}
    f(xy)+\chi(y)f(xy^{-1})=2f(x)f(y),\; x,y\in G,
\end{equation} where $\chi:$ $G\longrightarrow \mathbb{C}$ is a character of $G$. The non-zero solutions of equation (\ref{eq13}) are
the normalized traces of certain representation of the group $G$ on
$\mathbb{C}^{2}$\\ In the same year Stetk\ae r \cite{st4} obtained
the complex valued solution of the following variant   of
d'Alembert's functional equation
\begin{equation}\label{eq133}
    f(xy)+f(\sigma(y)x))=2f(x)f(y),\; x,y\in G,
\end{equation} where  $\sigma$ is an involutive homomorphism of $G$.  The  solutions of equation (\ref{eq133}) are
of the form $f(x)=\frac{\varphi(x)+\varphi(\sigma(x))}{2}$, $x\in
G$, where $\varphi$ is multiplicative.\\In \cite{eb1} Ebanks and
Stetk\ae r studied the solutions $f,g$: $G\longrightarrow
\mathbb{C}$ of Wilson's functional equation
\begin{equation}\label{eq14}
    f(xy)+f(xy^{-1})=2f(x)g(y),\; x,y\in G
\end{equation} and the following variant of Wilson's functional
equation (see \cite{st8})
\begin{equation}\label{eq15}
    f(xy)+f(y^{-1}x)=2f(x)g(y),\; x,y\in G.
\end{equation}They solve (\ref{eq15}) and they obtained some new
results about (\ref{eq14}). We refer also to Wilson's first
generalization of d'Alembert's functional equation:
\begin{equation}\label{eq166}
    f(x+y)+f(x-y)=2f(x)g(y),\;x,y\in \mathbb{R}.
\end{equation}For more about the functional equation (\ref{eq166})
see Acz\'el [\cite{ac1}, Section 3.2.1 and 3.2.2]. The solutions
formulas of
equation (\ref{eq166}) for abelian groups are known.\\
The stability of d'Alembert's functional equation was first
  obtained by Baker \cite{bak1} when the following theorem was proved.
  \begin{thm} \cite{bak1} (Superstability od d'Alembert's functional equation) Let $G$ be a group. If a function $f$: $G\longrightarrow \mathbb{{C}}$
  satisfies the inequality$$\mid f(x+y)+f(x-y)-2f(x)f(y)\mid\leq \delta$$
  for some $\delta>0$ and for all $x,y\in G$, then either $f$ is bounded on $G$ or $f(x+y)+f(x-y)=2f(x)f(y)$ for all $x,y\in G.$ \end{thm}
   A different generalization of Baker's result  was given by L.
Sz\'ekelyhidi \cite{z1,z2,z3}. It involves an interesting
generalization of the class of bounded function on a group or
semigroup. For other stability and superstability results, we can
see for example \cite{ba1}, \cite{ba2}, \cite{ba3}, \cite{el1},
\cite{el2}, \cite{ger1}, \cite{ger2} and \cite{el4}, the present
authors \cite{el3} for general groups.\\Various stability results of
Wilson's functional equation and it's generalization are obtained.
The number of papers in this subject is very high, hence, it is not
realistic to try to refer to all. The interested reader should refer
to \cite{ga1}, \cite{gav1}, \cite{for1}, \cite{h1}-\cite{h14} for a
thorough account on the subject of stability of functional
equations.\\ The main purpose of this paper is to study the
solutions and stability of the more general variant of Wilson's
functional equation
\begin{equation}\label{eq16}
    f(xy)+\chi(y)f(\sigma(y)x)=2f(x)g(y),\;x,y\in G,
\end{equation}where $G$ is a group, $\chi$ is a character
of $G$,  $\sigma$ is an involutive morphism of $G$. That is,
$\sigma(xy)=\sigma(y)\sigma(x)$ and $\sigma(\sigma(x))=x$ for all
$x,y\in G$ or $\sigma(xy)=\sigma(x)\sigma(y)$   and
$\sigma(\sigma(x))=x$ for all $x,y\in G$.\\
 We solve (\ref{eq16}) when $\sigma$ is an involutive
automorphism, and we obtain some properties of the solutions of
equation  (\ref{eq16}) when $\sigma$ is an involutive
anti-automorphism.  Furthermore,  we obtain the Hyers Ulam stability
of equation (\ref{eq16}). As an application we prove the
superstability of the functional equation
\begin{equation}\label{eq17}
     f(xy)+\chi(y)f(\sigma(y)x)=2f(x)f(y),\;x,y\in G.
\end{equation} \section{Stability of the functional equation (\ref{eq16}), where
$\sigma$ is an involutive anti-automorphism of $G$.}In this section
$\sigma$ is an involutive anti-automorphism  of $G$, that is
$\sigma(xy)=\sigma(y)\sigma(x)$ and $\sigma(\sigma(x))=x$, for all
$x,y\in G.$ The following theorem  is one of the main results of the
present paper.
\begin{thm} Let $\delta\geq0$. Let $\sigma$ be  an involutive anti-automorphism of $G$.
 Let $\chi$ be a unitary character of $G$ such that $\chi(x\sigma(x))=1$ for all $x\in G.$ Suppose that the  functions $f,g$:
 $G\longrightarrow \mathbb{C}$ satisfy the inequality
\begin{equation}\label{eq21} |f(xy)+\chi(y)f(\sigma(y)x)-2f(x)g(y)|\leq\delta
\end{equation} for all $x, y \in G.$  Under these assumptions the following statements hold: \\
(1) If $f$ is unbounded, then \\(i) $g$ is central. That is $g(xy) =
g(yx)$, for all $x,y\in G$;  $m_g$:
 $G\longrightarrow \mathbb{C}^{\ast}$ is multiplicative.
$$ (ii)\;\;\;\;
    g(x)=\chi(x)g(\sigma(x))\;\text{for all }\;x\in G,\; (\text{ if} \;\sigma(x)=x^{-1}, \;   \check{\chi}
m_{g}(G)\subseteq \{\mp 1\}).$$
$$ (iii)\;\;\;\;g(x)=m_{g}(x)g(x^{-1})\;\text{for all }\;x\in G$$
and \\(vi)
\begin{equation}\label{eq27}
g(xy)+m_{g}(y)g(xy^{-1})=2g(x)g(y)\;\text{for all }\;x,y\in G,
\end{equation}where $m_{g}(x)=2g(x)^{2}-g(x^{2})$, $x\in G$.\\
(2) If  $g$ is unbounded and $f\neq0$, then \\
(v) the pair ($f,g$) satisfies the functional equation (\ref{eq16}).
Furthermore, \\(vi) $m_g$:
 $G\longrightarrow \mathbb{C}^{\ast}$ is multiplicative, ( if $\sigma(x)=x^{-1}$,   $\check{\chi}
m_{g}(G)\subseteq \{\mp 1\}$).
$$ (vii) \;\;\;\;
g(x)=m_{g}(x)g(x^{-1}),\; g(x)=\chi(x)g(\sigma(x)),\;\\
\chi(y)f(\sigma(y)xy)=m_g(y)f(x)$$\text{for all}\;$x,y\in G$.\\
 (viii)
The pair $(f,g)$ satisfies
\begin{equation}\label{eq28}f(xy)+m_{g}(y)f(xy^{-1})=2f(x)g(y),\;x,y\in G
\end{equation} and
 $g$ satisfies equation (\ref{eq27}).
\end{thm}
\begin{proof}
All technical methods  that are needed in our discussion  are due to
Stetk\ae r \cite{st8}. (1) We let $L$ and $R$ denote respectively:
the left and right regular representation of $G$ on functions on
$G$. That is, $[L(y)h](x)=h(\sigma(y)x)$ and $R(y)h(x)=h(xy)$ for
$x,y\in G$ and $h$: $G\longrightarrow \mathbb{C}$. We notice here
that $L(x)R(y)=R(y)L(x)$ and $L(yz)h=L(y)[L(z)h]$,
$R(yz)h=R(y)[R(z)h]$ for all $x,y\in
G$ and for all function $h:$ $G\longrightarrow \mathbb{C}$.\\
Thus,  inequality (\ref{eq21}) can be written as follows
\begin{equation}\label{eq288}\|
[R(y)+\chi(y)L(y)]f-2g(y)f\|_{\infty}\leq \delta\end{equation} for
all $y\in G$. Applying the bounded operator $R(z)+\chi(z)L(z)$ to
the bounded function $R(y)+\chi(y)L(y)f-2g(y)f$ we get after
reduction that
\begin{equation}\label{eq211}
\|
(R(z)+\chi(z)L(z))[R(y)+\chi(y)L(y)]f-2g(y)(R(z)+\chi(z)L(z))f\parallel_{\infty}\end{equation}
$$=\|
[(R(zy)+\chi(zy)L(zy))f-2g(zy)f+2g(zy)f+\chi(y)R(z)L(z)+\chi(z)L(z)R(y)]f$$$$-2g(y)(R(z)+\chi(z)L(z)f-2g(z)f)-4g(z)g(y)f\|_{\infty}.$$
By using (\ref{eq288}), $\|(R(z)+\chi(z)L(z))h\|_{\infty}\leq2
\|h\|_{\infty}$ for all complex bounded function $h$ on $G$ we
obtain
\begin{equation}\label{eq212}
\|
2g(zy)f+[\chi(y)R(z)L(z)+\chi(z)L(z)R(y)]f-4g(z)g(y)f\|_{\infty}\leq3\delta+2|g(y)|\delta\end{equation}
for all $z,y\in G$. \\(1) (i) Interchanging $z$ and $y$ in
(\ref{eq212}) and substracting the result obtained from
(\ref{eq212})  we get
$$\|[g(zy)-g(yz)]f\|_\infty\leq 2|g(z)|\delta+2|g(y)|\delta+6\delta.$$
Since $f$ is assumed to be unbounded, then $g$ is central.
\\Setting $y=z$ in (\ref{eq212}), we obtain
\begin{equation}\label{eq213}
\|(2g(y)^{2}-g(y^{2}))f-\chi(y)R(y)L(y)f\|_\infty\leq |
g(y)|\delta+\frac{3}{2}\delta.
\end{equation}That is,
\begin{equation}\label{eq214}
|(2g(y)^{2}-g(y^{2}))f(x)-\chi(y)f(\sigma(y)xy)|\leq |
g(y)|\delta+\frac{3}{2}\delta
\end{equation} for all $x,y\in G$. Which implies that
\begin{equation}\label{eq01}
\|m_{g}(y)f-\chi(y)\mu(y)f\|_{\infty}\leq2|
g(y)|\delta+\frac{1}{2}\delta, \end{equation} where
$[\mu(y)h](x)=h(\sigma(y)xy).$ Noting that $\mu(yz)=\mu(y)\mu(z)$
for all $z,y\in G$. By using inequality (\ref{eq214}) we have
$$|m_g(yz)f(x)-\chi(y)\chi(z)f(\sigma(z)\sigma(y)xyz)|\leq|g(yz)|\delta+\frac{3}{2}\delta,$$
$$|m_g(y)f(x)-\chi(y)f(\sigma(y)xy)|\leq|g(y)|\delta+\frac{3}{2}\delta,$$
and
 $$|m_g(z)f(\sigma(y)xy)-\chi(z)f(\sigma(z)\sigma(y)xyz)|\leq|g(z)|\delta+\frac{3}{2}\delta.$$ So, by using triangle inequality we get
 $$|m_g(yz)f(x)-m_g(y)m_g(z)f(x)|\leq |m_g(yz)f(x)-\chi(y)\chi(z)f(\sigma(z)\sigma(y)xyz)|$$
 $$+|\chi(y)\chi(z)f(\sigma(z)\sigma(y)xyz)-m_g(y)m_g(z)f(x)|$$
 $$\leq|m_g(yz)f(x)-\chi(y)\chi(z)f(\sigma(z)\sigma(y)xyz)|$$
 $$+|m_g(z)\chi(y)f(\sigma(y)xy)-m_g(y)m_g(z)f(x)|+|\chi(y)\chi(z)f(\sigma(z)\sigma(y)xyz)-m_g(z)\chi(y)f(\sigma(y)xy)|$$
 $$\leq |g(yz)|\delta+\frac{3}{2}\delta+|m_g(z)||\chi(y)f(\sigma(y)xy)-m_g(y)f(x)|$$$$+|\chi(y)||\chi(z)f(\sigma(z)\sigma(y)xyz)-m_g(z)f(\sigma(y)xy)|$$
 $$\leq|g(yz)|\delta+\frac{3}{2}\delta+|m_g(z)|[|g(y)|\delta+\frac{3}{2}\delta]+|\chi(y)|[|g(z)|\delta+\frac{3}{2}\delta].$$ From the assumption that $f$ is unbounded we get
$m_{g}(yz)=m_g(y)m_g(z)$ for all $y,z\in G$. On the other hand if
$m_g=0$, then if we put $y=e$ in  (\ref{eq214})  we obtain $f$
bounded, since
$f$ is unbounded, so  $m_g(G)\subseteq\mathbb{C}^{\ast}$. \\
(ii) Now, let $a\in G$ be arbitrary. First case: Assume that either
$f(a)\neq 0$ or $f(e)\neq 0$. The pair $(f,g$) satisfies inequality
(\ref{eq21}) on the abelian subgroup $<a>$ generated by $a$, then on
the abelian subgroup $<a>$ we have
$|f(x\sigma(y))+\chi(\sigma(y))f(xy)-2f(x)g(\sigma(y))|\leq\delta$,
since $\chi$ is unitary and $\chi(y\sigma(y))=1$ hence we get
\begin{equation}\label{eq215}
|\chi(y)f(x\sigma(y))+f(xy)-2f(x)\chi(y)g(\sigma(y))|\leq\delta,\;x,y\in
G.\end{equation} By substituting (\ref{eq21}) to ($\ref{eq215}$) on
the commutative subgroup $<a>$ we obtain
$|f(x)[g(y)-\chi(y)g(\sigma(y))]|\leq2\delta$ for all $x,y\in G$.
Since $f$ is unbounded, then we have $g(y)=\chi(y)g(\sigma(y))$ for
all $y\in <a>.$ In particular $g(a)=\chi(a)g(\sigma(a))$
\\Second case: Assume that $f(a)= 0$ and $f(e)=0$. Setting $x=e$
in (\ref{eq21}), we obtain
\begin{equation}\label{eq216}
|f(y)+\chi(y)f(\sigma(y))|\leq\delta\end{equation} for all $y\in G$.
Thus, from $\chi(a\sigma(a))=1$,  $2f(x)[g(a)-\chi(a)g(\sigma(a))]$
can be written as follows
$$|2f(x)[g(a)-\chi(a)g(\sigma(a))]=2f(x)g(a)-f(xa)-\chi(a)f(\sigma(a)x)$$$$
+\chi(a)[f(x\sigma(a))+\chi(\sigma(a))f(ax)-2f(x)g(\sigma(a))]$$$$+f(xa)+\chi(a)f(\sigma(a)x)-\chi(a)f(x\sigma(a))-f(ax).$$
Since
$$f(xa)+\chi(a)f(\sigma(a)x)-\chi(a)f(x\sigma(a))-f(ax)$$$$=f(xa)+\chi(a)f(\sigma(a)x)-\chi(a)[f(x\sigma(a))+\chi(x\sigma(a))f(a\sigma(x))]+\chi(x)f(a\sigma(x))$$
$$-[f(ax)+\chi(ax)f(\sigma(x)\sigma(a))]+\chi(ax)f(\sigma(x)\sigma(a))$$
$$=\chi(x)[f(a\sigma(x))+\chi(\sigma(x))f(\sigma(\sigma(x))a)-2f(a)g(\sigma(x))]+2\chi(x)f(a)g(\sigma(x))
$$$$+\chi(a)[f(\sigma(a)x)+\chi(x)f(\sigma(x)\sigma(a))-2f(\sigma(a))g(x)]+2\chi(a)f(\sigma(a))g(x)$$$$-\chi(a)[f(x\sigma(a))+\chi(x\sigma(a))f(a\sigma(x))]
-[f(ax)+\chi(ax)f(\sigma(x)\sigma(a))].$$ From inequalities
(\ref{eq216}), (\ref{eq21}) and $|\chi(x)|=1$ we get
\begin{equation}\label{eq217}
|2f(x)[g(a)-\chi(a)g(\sigma(a))]|\leq6\delta+2|f(\sigma(a))|g(x)|+2|f(a)||g(\sigma(x))|\end{equation}
for all $x\in G$. Here again we discuss two subcases: If $g$ is
bounded, then by using the unboundedness of $f$ and (\ref{eq217}) we
get $g(a)=\chi(a)g(\sigma(a))$. If $g$ is unbounded we use the case
$ii)$ to obtain that $g(x)=\chi(x)g(\sigma(x))$ for all $x\in G$
hence we get the result for $x=a$. On the other hand we have
$m_g(\sigma(y))=2g(\sigma(y))^{2}-g(\sigma(y^{2}))=
2\chi(y)^{2}g(y)^{2}-\chi(\sigma(y))^{2}g(y^{2})=\chi(y)^{2}m_g(y)$.
So if $\sigma(x)=x^{-1}$ for all $x\in G$ then we get
$m_g(x)m_g(x^{-1})=m_g(e)=1=m_g(x)\chi(x^{-1})^{2}m_g(x)=(\chi(x^{-1})m_g(x))^{2}$,
then we get $\check{\chi}m_g(G)\subseteq \{\mp 1\}$.
\\
(iii) The formula $ 2f(x)[g(y)-m_g(y){\check{g}}(y)]$ can be written
as follows
$$2f(x)[g(y)-m_g(y){\check{g}}(y)]=$$$$[-(f(xy)+\chi(y)f(\sigma(y)x))+2f(x)g(y)]$$$$+m_g(y)[f(xy^{-1})+\chi(y^{-1})f(\sigma(y^{-1})x))-2f(x){\check{g}}(y)]$$
$$+f(xy)-\chi(y^{-1})m_g(y)f(\sigma(y^{-1})x)+\chi(y)f(\sigma(y)x)-m_g(y)f(xy^{-1}).$$
Once again we have
$$f(xy)-\chi(y^{-1})m_g(y)f(\sigma(y^{-1})x)+\chi(y)f(\sigma(y)x)-m_g(y)f(xy^{-1})$$
$$=f(\sigma(y)\sigma(y^{-1})xy)-\chi(y^{-1})m_g(y)f(\sigma(y^{-1})x)+\chi(y)f(\sigma(y)xy^{-1}y)-m_g(y)f(xy^{-1})$$
$$=[\mu(y)f](\sigma(y^{-1})x)-\chi(y^{-1})m_g(y)f(\sigma(y^{-1})x)+\chi(y)[\mu(y)f](xy^{-1})-m_g(y)f(xy^{-1})$$
and from inequality (\ref{eq214}) we obtain
\begin{equation}\label{eq218}
|[\mu(y)f](\sigma(y^{-1})x)-\chi(y^{-1})m_g(y)f(\sigma(y^{-1})x))+\chi(y)[\mu(y)f](xy^{-1})-m_g(y)f(xy^{-1})|\end{equation}$$\leq2|g(y)|\delta+3\delta.$$
Now, from inequalities  (\ref{eq21}) and (\ref{eq218}) we get
\begin{equation}\label{eq3000}
|2f(x)[g(y)-m_g(y){\check{g}}(y)]|\leq
|m_g(y)|\delta+2|g(y)|\delta+4\delta.\end{equation}
 Since $f$ is
unbounded then we have $g(y)=m_g(y){\check{g}}(y)$ for all $y\in G.$
\\(vi) Let us consider
\begin{equation}\label{eq219}
2f(x)[g(zy)+m_g(y)g(zy^{-1})-2g(y)g(z)]\end{equation}

$$=[2g(zy)f(x)-4g(y)g(z)f(x)+\chi(y)(R(z)L(y)f)(x)+\chi(z)(L(z)R(y)f)(x)]$$$$+2f(x)m_g(y)g(zy^{-1})-\chi(y)(R(z)L(y)f)(x)-\chi(z)(L(z)R(y)f)(x).$$
Since
$$2f(x)m_g(y)g(zy^{-1})-\chi(y)(R(z)L(y)f)(x)-\chi(z)(L(z)R(y)f)(x)$$
$$=2f(x)m_g(y)g(zy^{-1})-\chi(y)f(\sigma(y)xzy^{-1}y)-\chi(z)f(\sigma(y)\sigma(y^{-1})\sigma(z)xy)$$
$$=2f(x)m_g(y)g(zy^{-1})-\chi(y)f(\sigma(y)xzy^{-1}y)+m_g(y)f(xzy^{-1})-m_g(y)f(xzy^{-1})$$
$$-\chi(z)f(\sigma(y)\sigma(y^{-1})\sigma(z)xy)+\chi(y^{-1})\chi(z)m_g(y)f(\sigma(y^{-1})\sigma(z)x))$$
$$-\chi(y^{-1})\chi(z)m_g(y)f(\sigma(y^{-1})\sigma(z)x))$$
$$=-\chi(y)[\mu(y)f](xzy^{-1})+m_g(y)f(xzy^{-1})$$$$-\chi(z)[\mu(y)f](\sigma(y^{-1})\sigma(z)x)+\chi(y^{-1})m_g(y)f(\sigma(y^{-1})\sigma(z)x))$$
$$-m_g(y)[f(xzy^{-1})+\chi(zy^{-1})f(\sigma(y^{-1})\sigma(z)x)-2f(x)g(zy^{-1})].$$
From inequalities (\ref{eq21}), (\ref{eq214}), (\ref{eq215}) and
(\ref{eq219}) we obtain
$$|2f(x)[g(zy)+m_g(y)g(zy^{-1})-2g(y)g(z)]|\leq 6\delta+4|g(y)|\delta+|m_g(y)|\delta.$$
Since $f$ is unbounded, then $g$ satisfies  equation
$g(xy)+m_g(y)g(xy^{-1})=2g(x)g(y),\;x,y\in G.$ \\(2)  We assume that
$g$ is unbounded and $f\neq0$. By simple computations we get $f$
unbounded. Now, for all $x,y,z\in G$ we have
$$|2g(z)||f(xy)+\chi(y)f(\sigma(y)x)-2f(x)g(y)|$$
$$=|2f(xy)g(z)+2\chi(y)f(\sigma(y)x)g(z)-4f(x)g(y)g(z)|$$
$$\leq|-f(xyz)-\chi(z)f(\sigma(z)xy)+2f(xy)g(z)|$$
$$+|\chi(y)[-f(\sigma(y)xz)-\chi(z)f(\sigma(z)\sigma(y)x)+2f(\sigma(y)x)g(z)]|$$
$$+|+f(xyz)+\chi(yz)f(\sigma(z)\sigma(y)x)-2f(x)g(yz)|$$
$$+|-f(xzy)-\chi(zy)f(\sigma(y)\sigma(z)x)+2f(x)g(zy)|$$
$$+|\chi(z)f(\sigma(z)xy)+\chi(zy)f(\sigma(y)\sigma(z)x)-2\chi(z)f(\sigma(z)x)g(y)|$$
$$+|f(xzy)+\chi(y)f(\sigma(y)xz)-2f(xz)g(y)|$$
$$+2|f(x)||g(yz)-g(zy)|+2g(y)|f(xz)+\chi(z)f(\sigma(z)x)-2f(x)g(z)|.$$
$$\leq\delta+\delta+\delta+\delta+|\chi(z)|\delta+0\times  2|f(x)|\delta+2|g(y)|\delta=6\delta+2|g(y)|\delta$$
Using that $g$ is unbounded, we get the desired result that the pair
($f,g$) satisfies the functional equation (\ref{eq16}). Now, by
using (\ref{eq214}) with $\delta=0$ we get
\begin{equation}\label{eq90}
m_g(y)f(x)=\chi(y)f(\sigma(y)xy)\end{equation}  for all $x,y\in G.$
So if we replace $x$ by $xy^{-1}$ in(\ref{eq90}) we obtain
$m_g(y)f(xy^{-1})=\chi(y)f(\sigma(y)x)$ and equation (\ref{eq16})
can be written as follows $f(xy)+m_g(y)f(xy^{-1})=2f(x)g(y)$ for all
$x,y\in G.$ For the proof of other properties we use case (1) with
$\delta=0.$ This completes the proof of theorem. \end{proof}As an
application we get some properties of the solutions of equation
(\ref{eq16}), where $\sigma$ is an involutive anti-automorphism. By
using the above Theorem for $\delta=0$, [\cite{eb1}, Proposition 2,
Theorem 3 and Corollary 6], [\cite{st2}, Theorem 6.1, Theorem 10.1 and Corollary 10.2] we obtain the following theorem.\\
For later use, we recall (see for example \cite{eb1}) that a
function $f$: $G\longrightarrow \mathbb{C}$ is said to be abelian,
if $f(x_{\sigma(1)}x_{\sigma(2)}...x_{\sigma(n)})=f(x_1x_2...x_n)$
for all $x_1,x_2,...,x_n\in G$, all permutations $\sigma$ and all
$n=1,2,...$.
\begin{thm}Let the pair $f,g$: $G\longrightarrow \mathbb{C}$ be a  solution of the variant (\ref{eq16}) of Wilson's functional equation such that $f\neq 0$.\\
\textbf{(1)} If $f$ is a nonzero central function. Then, \\(i)
$f=f(e)g$, when $g$ is non abelian.\\(ii) When $g$ is abelian $g$
has the form $g=\frac{\psi+\chi\psi\circ\sigma}{2}$ where $\psi:$
$G\longrightarrow \mathbb{C}^{\ast}$ is a character. If
$\psi\neq\chi\psi\circ\sigma$, then
$f=\alpha(\psi-\chi\psi\circ\sigma)/2+\beta(\psi+\chi\omega\circ\sigma)/2$
for some $\alpha,\beta\in \mathbb{C}$. If
$\psi=\chi\psi\circ\sigma$, then $f=\psi a+\beta\psi$ for some
additive map $a$: $G\longrightarrow \mathbb{C}$ and some $\beta\in
\mathbb{C}$.
  \\\textbf{(2)} (i) $g(e)=1$, $g$ is central, and
$g=\chi g\circ\sigma$, $g=m_g \check{g}$.\\
(ii) $m_g$: $G\longrightarrow \mathbb{C}^{\ast}$ is multiplicative.\\
(iii) $\chi(y)f(\sigma(y)xy)=m_g(y)f(x)$, $\chi(y)f(\sigma(y)x)=m_g(y)f(xy^{-1})$ for all $x,y\in G.$\\
(iv) $g(xy)+m_g(xy^{-1})=2g(x)g(y)$ for all $x,y\in G.$\\
(v) $f(xy)+m_g(y)f(xy^{-1})=2f(x)g(y)$ for all $x,y\in G.$\\
(vi) $f=-\chi f\circ\sigma$ if and only if $f(e)=0$\\
(vii) The even part and the odd part of $f$:
$f_e(x)=\frac{f(x)+\chi(x)f(\sigma(x))}{2}$,
$f_o(x)=\frac{f(x)-\chi(x)f(\sigma(x))}{2}$ and $\chi f\circ\sigma$
satisfy
(\ref{eq16}) with $g$ unchanged.\\
(ix) $f_e=f(e)g$. In particular $f_e=0$ (f is odd) if and only if $f(e)=0$  \\
(x) The odd part $f_o$ of $f$ satisfies
\begin{equation}\label{eq30000}
f_o(xy)+f_o(yx)=2f_o(x)g(y)+2f_o(y)g(x)
\end{equation} for all $x,y\in G.$.\\
\textbf{(3)} For the rest we assume that $\sigma(x)=x^{-1}$ for all
$x\in
G.$\\
(i)  $\check{\chi} m_{g}(G)\subseteq \{\mp 1\}$\\
 If $m_g=\chi$,
then either\\
(i) $g$ is non-abelian and $f=f(e)g$, or \\
(iii) $f$ and $g$ are both abelian, in which case (1) applies.\\
\textbf{(4)} If $m_g\neq\chi$, then $f=-\chi f\circ\sigma$ (f is
odd).
\end{thm}
\section{Solutions and Stability of the functional equation
(\ref{eq16}), where $\sigma$ is an involutive homomrphism of $G$} In
this section $\sigma$ is an involutive  homomorphism of $G$, that is
$\sigma(xy)=\sigma(x)\sigma(y)$ and $\sigma(\sigma(x))=x$, for all
$x,y\in G.$ In the following theorem, we obtain the solutions of the
functional  equation (\ref{eq16}) on semigroups with  identity
element. It turns out that, like on abelian groups, only
multiplicative and additive functions occur in the solution
formulas.
\begin{thm}Let $G$ be a semigroup with  identity element,  $\sigma$: $G\longrightarrow G$ a multiplicative function such that $\sigma\circ\sigma=I$,
 where $I$ denotes the identity map, and $\chi$: $G\longrightarrow \mathbb{C}$ be a character of $G$ such that $\chi(x\sigma(x))=1$ for all
$x\in G$. \\
The solutions $f,g$ of the functional equation (\ref{eq16}) are the
following pairs of functions, where $m$: $G\longrightarrow
\mathbb{C}$ denotes a function multiplicative and $c\in \mathbb{C}^{\ast}$.\\
(i) $f=0$ and $g$ arbitrary.\\
(ii) $g=\frac{m+\chi m\circ\sigma}{2}$ and $f=f(e)g$. \\
(iii) $g=\frac{m+\chi m\circ\sigma}{2}$ and
$f=(c+\frac{f(e)}{2})m-(c-\frac{f(e)}{2})m\circ\sigma$ with $(\chi-1)m=(\chi-1)m\circ\sigma.$ \\
(iv) $g=m$ and $f=(a+f(e))m$, where $m=\chi m\circ\sigma$ and $a$:
$G\longrightarrow \mathbb{C}$ is an additive map which satisfies
$m(a\circ\sigma+a)=0$.
\end{thm}\begin{proof} It is elementary to check that the cases stated in the Theorem define solutions, so it is left to show that any solution $f,g$: $G\longrightarrow \mathbb{C}$
 of (\ref{eq16}) falls into one of these cases. We use in the proof similar Stetk\ae r's computations  \cite{st4} . Let $x,y,z\in
G$. If we replace $x$ by $xy$ and $y$ by $z$ in (\ref{eq16}) we get
\begin{equation}\label{eq31}
f(xyz)+\chi(z)f(\sigma(z)xy)=2f(xy)g(z).
\end{equation}
On the other hand if we replace $x$ by $\sigma(z)x$ in (\ref{eq16}),
we obtain
$$f(\sigma(z)xy)+\chi(y)f(\sigma(y)\sigma(z)x)=2f(\sigma(z)x)g(y)$$
$$=2g(y)[\chi(\sigma(z))[2f(x)g(z)-f(xz)]].$$
Since,
$$\chi(y)f(\sigma(y)\sigma(z)x)=\chi(y)f(\sigma(yz)x)=\chi(y)\chi(\sigma(yz))[2g(yz)f(x)-f(xyz)]$$
$$=\chi(\sigma(z))[2g(yz)f(x)-f(xyz)],$$ so by using $\chi(z\sigma(z))=1$ we have
 \begin{equation}\label{eq32}
 \chi(z)f(\sigma(z)xy)+[2g(yz)f(x)-f(xyz)]=2g(y)[2f(x)g(z)-f(xz).\end{equation}
 Subtracting this from (\ref{eq31}) we get
\begin{equation}\label{eq3333}
    f(xyz)=g(yz)f(x)+f(xy)g(z)+g(y)f(xz)-2g(y)f(x)g(z).
\end{equation} With the notation
\begin{equation}\label{eq331}
f_{x}(y)=f(xy)-f(x)g(y)\end{equation}
 equation (\ref{eq3333}) can be
written as follows
\begin{equation}\label{eq34}
f_{a}(xy)=f_{a}(x)g(y)+f_{a}(y)g(x),\; x,y\in G.
\end{equation}We will in the rest of the proof of Theorem 3.1 need to know the solutions of the
functional equation
\begin{equation}\label{eq35}
f(xy)+\chi(y)f(\sigma(y)x)=2f(x)f(y);\;x,y\in G.
\end{equation} They are obtained in the following lemma. \begin{lem} Let $G$ be a semigroup with  identity element,
$\sigma$: $G\longrightarrow G$ a multiplicative function such that
$\sigma\circ\sigma=I$,
 where $I$ denotes the identity map, and $\chi$: $G\longrightarrow \mathbb{C}$ be a character of $G$ such that $\chi(x\sigma(x))=1$ for all
$x\in G$.  The solutions $f$ of the functional equation (\ref{eq35})
are of the form $f=\frac{m+\chi m\circ\sigma}{2}$, where $m$:
$G\longrightarrow \mathbb{C}$ is multiplicative.\end{lem}
\begin{proof} Verifying that  $f=\frac{m+\chi m\circ\sigma}{2}$, where $m$:
$G\longrightarrow \mathbb{C}$ is multiplicative, is solution of
equation (\ref{eq35}) consists in simple computations. Let $f$
satisfies the functional equation (\ref{eq35}), then by using the
above computations the pair $f,f_a$ satisfies equation
\begin{equation}\label{eq36}
    f_{a}(xy)=f_{a}(x)f(y)+f_{a}(y)f(x),\; x,y\in G.
\end{equation}
If $f_a=0$ for all $a\in G$ then $f$ is multiplicative. Substituting
$f$ in (\ref{eq35}) we get $\chi(y)f(\sigma(y))=f(y)$ for all $y\in
G$. This implies that $f=\frac{\varphi+\chi\varphi}{2}$, where
$f=\varphi$ is
multiplicative.\\
If $f_a\neq 0$ for some $a\in G$. From the known solution of the
sine addition formula (see for example [\cite{st3}, Theorem 4.1])
there exist two multiplicative functions $\chi_1,\chi_2$:
$G\longrightarrow \mathbf{C}$ such that $f=\frac{\chi_1+\chi_2}{2}$.
We can assume that $\chi_1\neq\chi_2$. Substituting
$f=\frac{\chi_1+\chi_2}{2}$ in (\ref{eq35}) we get after reduction
that
$$\chi_1(x)[\chi(y)\chi_1(\sigma(y))-\chi_2(y)]=\chi_2(x)[\chi_1(y)-\chi(y)\chi_2(\sigma(y))].$$
Since $\chi_1\neq\chi_2$ at least one of $\chi_1$ and $\chi_2$ is
not zero. So, we get $\chi_1=\chi\chi_2\circ\sigma$, and
$f=\frac{\varphi+\chi\varphi\circ\sigma}{2}$, where $\varphi$:
$G\longrightarrow \mathbb{C}$ is multiplicative. This completely
describes  the solutions of equation (\ref{eq35}).\end{proof} Now,
we will find the solutions of equation (\ref{eq16}). Let $f,g$:
$G\longrightarrow \mathbb{C}$ solution of equation (\ref{eq16}). The
above computation show that the pair $f_a,g$ satisfies the sine
addition formula (\ref{eq34}) for any $a\in G.$ From the known
solution of the sine addition formula (see for example [\cite{st3},
Theorem
4.1]) we have the following possibilities.\\
If $f=0$ we deal with case (i). So during the rest of the proof we
will assume that $f\neq 0$. If we replace $a$ by $e$ in
(\ref{eq331}) we get $f_{e}(x)=f(x)-f(e)g(x)$. If $f_{e}=0$, then
$f(x)=f(e)g(x)$ for all $x\in G$. Since $f\neq0$ then $f(e)\neq0$.
Substituting $f=f(e)g$ into (\ref{eq16}) we find that $g$ satisfies
equation (\ref{eq35}) then there exists $m$: $G\longrightarrow
\mathbb{C}$ multiplicative such that $g=\frac{m+\chi
m\circ\sigma}{2}$. We see that
we deal with case (ii). \\
If $f_e\neq 0$, the pair $(f_e,f)$ satisfies (\ref{eq34}) and  we
known from [\cite{st3}, Theorem 4.1] that there are only the
following
3 possibilities:\\
(1) $f_e=cm$ and $g=m/2$ for some $m$ multiplicative. Here
$f=f_e+f(e)g$. Substituting $f=(c+\frac{f(e)}{2})m$, $g=m/2$ into
(\ref{eq16}) we find $(c+\frac{f(e)}{2})\chi(y)m(x)m(\sigma(y))=0$
for all $x,y\in G.$ This case does not apply, because $f\neq 0.$\\
(2) There exist two different characters $m$ and $M$ on $G$ and a
constant $c\in \mathbb{C}^{\ast}$ such that
$$g=\frac{m+M}{2}\; \text{and}\;f_e=c(m-M)$$
then $f=c(m-M)+f(e)\frac{m+M}{2}=\alpha m-\beta M$, where
$\alpha=c+\frac{f(e)}{2}$ and $\beta=c-\frac{f(e)}{2}$. Substituting
this into (\ref{eq16}) we find after reduction that
\begin{equation}\label{eq37}
    \alpha m(x)(\chi(y)m(\sigma(y))-M(y))=\beta
    M(x)(\chi(y)M(\sigma(y))-m(y)).
\end{equation}
If we replace $y$ by $\sigma(y)$ in (\ref{eq37}) and after we
multiply equation obtained by $\chi(y)$ and using
$\chi(y\sigma(y))=1$ we find
\begin{equation}\label{eq38}
    \alpha m(x)(m(y)-\chi(y)M(\sigma(y)))=\beta
    M(x)(M(y))-\chi(y)m(\sigma(y))).
\end{equation}Subtracting (\ref{eq37}) from  (\ref{eq38}) we get
after some simplifications that
\begin{equation}\label{eq39}
    (\alpha m(x)+\beta M(x))(\chi(y)m(\sigma(y))-M(y))=(\alpha
    m(x)+\beta M(x))(\chi(y)M(\sigma(y))-m(y)).
\end{equation} Putting $x=e$ in (\ref{eq39}) we find that
$\chi(y)m(\sigma(y))-M(y)=\chi(y)M(\sigma(y))-m(y)$, because
$\alpha+\beta=2c\neq0.$ If $\chi M\circ\sigma-m\neq0$, then from
(\ref{eq37})we get $\alpha m(x)=\beta M(x)$ for all $x\in G$. So,
for $x=e$ we obtain $\alpha=\beta$ which contradicts the assumption
that $f(e)\neq0$. Thus,
 $M=\chi m\circ\sigma$ and $m=\chi M\circ\sigma$ from which we see
that $g=\frac{m+\chi m\circ\sigma}{2}$ and
$f=(c+\frac{f(e)}{2})m-(c-\frac{f(e)}{2})m\circ\sigma$. We conclude
that we deal with case (iii).\\
(3)  $g=m$ and $f_e=ma$, where $m$ is multiplicative of $G$ and $a$
is an additive map. From $f_e=f-f(e)g$ we get
$f=ma+f(e)m=(a+f(e))m$. Substituting this into (\ref{eq16}) we find
after reduction that
\begin{equation}\label{eq310}
    m(x)(a(y)m(y)+\chi(y)a(\sigma(y))m(\sigma(y)))+m(x)(a(x)+f(e))(\chi(y)m(\sigma(y))-m(y))=0.
\end{equation}
If we replace $y$ by $\sigma(y)$ in (\ref{eq310})    and after we
multiply equation obtained by $\chi(y)$ and using
$\chi(y\sigma(y))=1$ we find
\begin{equation}\label{eq311}
    m(x)(\chi(y)a(\sigma(y))m(\sigma(y))+a(y)m(y))+m(x)(a(x)+f(e))(m(y)-\chi(y)m(\sigma(y)))=0.
\end{equation}
Subtracting (\ref{eq310}) from  (\ref{eq311}) we get after some
simplifications that
\begin{equation}\label{eq312}
    2m(x)(a(x)+f(e))(\chi(y)m(\sigma(y))-m(y))=0
\end{equation} for all $x,y\in G.$ Putting $x=e$ in (\ref{eq312}) we
get $m=\chi m\circ\sigma$, because
$2m(e)(a(e)+f(e))=2.1.(0+f(e))=2f(e)\neq0.$ This means that
$g=\frac{m+\chi m\circ\sigma}{2}$. Substituting  $m=\chi
m\circ\sigma$ into (\ref{eq310}) we deduce that
$m(a\circ\sigma+a)=0$. We see that we deal with case (iv) and this
completes the proof.
\end{proof} The formulas of
Theorem 3.1 implies  the following corollary.
\begin{cor} Let $G$ be a semigroup with  identity element,  $\sigma$: $G\longrightarrow G$ a multiplicative function such that $\sigma\circ\sigma=I$,
 where $I$ denotes the identity map, and $\chi$: $G\longrightarrow \mathbb{C}$ be a multiplicative function of $G$ such that $\chi(x\sigma(x))=1$ for all
$x\in G$.  If $f; g$ : $G\longrightarrow \mathbb{C}$ is a solution
of variant of Wilson's functional equation (\ref{eq16})  such that
$f\neq 0$, then $g$ is a solution of variant of d'Alembert's
functional equation (\ref{eq35}).\end{cor} In the rest of this
section we examine the Hyers-Ulam stability of the functional
equation (\ref{eq16}). We shall first recall two variants of
Sz\'ekelyhidi results because it will be useful in the treatment of
stability of other functional equations like sine addition formula.
The proof of Theorem 3.3 and Theorem 3.4 goes along  the same lines
as the one in \cite{z1} and   \cite{z2}.
\begin{thm} \cite{z1} Let $V$ be a vector space of $\mathbb{C}$-valued functions on a semigroup $G$, let $V$ be left invariant and suppose that $f$ and $m$ are $\mathbb{C}$-valued functions on $G$.
 If the function $y\longmapsto f(xy)-f(y)m(x)$ belongs to $V$ for all $x\in G$. Then either $f$ is in $V$ or  $m$ is an exponential.  \end{thm}
\begin{thm} \cite{z4} Let $V$ be a vector space of $\mathbb{C}$-valued functions on a semigroup $G$, let $V$ be  invariant and suppose that $f$ and $g$ are $\mathbb{C}$-valued functions on
$G$ which are linearly independent modulo $V$.
 If the function $x\longmapsto f(yx)-f(x)g(y)-f(y)g(x)$ belongs to $V$ for all $y\in G$, then  $f(xy)=f(x)g(y)+f(y)g(x)$ for all $x,y\in G$.  \end{thm}
 In the following theorem we obtain the Hyers-Ulam stability of the
 functional equation (\ref{eq16}). The following lemmas will be helpful in the sequel.
 \begin{lem} Let $\delta\geq0$, let $G$ be a semigroup with  identity element,  $\sigma$: $G\longrightarrow G$ is an homomorphism  such that $\sigma\circ\sigma=I$,
 and $\chi$: $G\longrightarrow \mathbb{C}$ be a bounded multiplicative function
such that $\chi(x\sigma(x))=1$ for all $x\in G$. Suppose that the
pair $f,g:G\rightarrow \mathbb{C}$ satisfies
\begin{equation}\label{eq400}
|f(xy)+\chi(y)f(\sigma(y)x)-2f(x)g(y)|\leq \delta,\;\text{for all
}\;x,y\in G.
\end{equation} Under these assumptions the following statement
hold:\begin{equation}\label{eq401}
|f_a(xy)-f_a(x)g(y)-f_a(y)g(x)|\leq
|g(x)|\delta+\frac{3}{2}\delta,\;\text{for all }\;x,y\in
G,\end{equation}\end{lem} where  $f_a$ is the function defined in
(\ref{eq331}).
\begin{proof}For $x,y\in G$ we put $F(x,y)=f(xy)+\chi(y)f(\sigma(y)x)-2f(x)g(y)$. By using similar computations used in the proof of
Theorem 3.1  we get
\begin{equation}\label{eq402}
    f(xyz)-f(x)g(yz)+2f(x)g(y)g(z)-f(xy)g(z)-g(y)f(xz)\end{equation}$$=-g(y)F(x,z)+\frac{F(x,yz)}{2}+\frac{F(xy,z)}{2}-\frac{F(\sigma(z)x,y)}{2}$$
for all $x,y,z\in G.$ From inequality (\ref{eq400}) and the
definition of $f_a$ we get the desired result.
\end{proof} The second main result of this section  is the next one.
\begin{thm} Let $G$ be a group with  identity element,  $\sigma$: $G\longrightarrow G$ an involutive homomorphism
of $G$
 and $\chi$: $G\longrightarrow \mathbb{C}$ be a unitary character of $G$
such that $\chi(x\sigma(x))=1$ for all $x\in G$. Let the pair
$f,g:G\rightarrow \mathbb{C}$ be given.  Suppose that the function
\begin{equation}\label{eq313}(x,y)\longrightarrow
f(xy)+\chi(y)f(\sigma(y)x)-2f(x)g(y)\end{equation} is bounded. Under
these assumptions the following statements
hold:\\
(i) $f=0$ and $g$ arbitrary.\\
(ii) $f\neq 0$ is bounded and $g$ is bounded.\\
(iii) $f$ is unbounded, $g$ is bounded and $G$ is an amenable group,
then $g\neq 0$ is multiplicative, $g=\chi g\circ\sigma$ and there
exists an additive map $a$: $G\longrightarrow \mathbb{C}$ such that
$f-ag$ is bounded and
$(ag)(xy)+\chi(y)(ag)(\sigma(y)x)=2(ag)(x)g(y)$ for all $x,y\in G.$
\\
(iv) $f$ is unbounded, $g$ is unbounded. In this case there are the
following three possibilities:\\(1) $g$ is multiplicative, $g=\chi
g\circ\sigma$, $f=f(e)g.$
 Furthermore, $f,g$ satisfy equation  (\ref{eq16}).\\
(2) $g$ is multiplicative, $g=\chi g\circ\sigma$, $f=ag$, where  $a$ is an additive map such that $a\circ\sigma=-a$. Furthermore, $f,g$ satisfy equation  (\ref{eq16}).\\
(3) $g=\frac{m+\chi m\circ\sigma}{2}$ and
$f=(c+\frac{f(e)}{2})m-(c-\frac{f(e)}{2})m\circ\sigma,$ where $m$ is
multiplicative.
\end{thm}\begin{proof}If $f=0$ we deal with case (i). So during the rest of the proof we will assume that $f\neq 0$.
If $f$ is bounded then by using (\ref{eq313}) we get $g$ bounded. This is case (ii)\\
(iii) If $f$ is unbounded and $g$ bounded. We notice here that
$g\neq 0$, because if $g=0$ then from (\ref{eq313}) with $y=e$ we
get $f$ bounded, which contradict our assumption that $f$ is
unbounded. We put $h=f-f(e)$, so $h(e)$=0 and the function
\begin{equation}\label{eq314}(x,y)\longrightarrow
h(xy)-h(x)g(y)-h(y)g(x)\end{equation} is bounded. Thus, the function
$y\longmapsto h(xy)-h(y)g(x)$ is bounded for all $x\in G$. So, by
using Theorem 3.4 we get $g=m$
 multiplicative and the function defined in (\ref{eq314}) remains bounded when the right side is multilied by $m((xy)^{-1})=m(x^{-1})m(y^{-1})$, so that the function $(x,y)\longrightarrow
 \frac{h}{m}(xy)-\frac{h}{m}(x)-\frac{h}{m}(y)$ is bounded. Since
 $G$ is amenable then from \cite{z5}
 we have $\frac{h}{m}(x)=a(x)+b(x)$  for all $x\in G$,  where  $a$: $G\longrightarrow \mathbb{C}$ is an unbounded additive map and
 $b$: $G\longrightarrow \mathbb{C}$ is bounded. On the other hand by Substituting this into
 $B(x,y)=
 h(xy)+\chi(y)h(\sigma(y)x)-2h(x)g(y)$ we get
 $a(xy)m(xy)+b(xy)m(xy)+\chi(y)[a(\sigma(y)x)m(\sigma(y)x)+b(\sigma(y)x)m(\sigma(y)x)]=2[a(x)m(x)+b(x)m(x)]m(y)+B(x,y)$
 and we find after reduction that the function
 \begin{equation}\label{eq315}
   |a(x)m(x)(\chi(y)m(\sigma(y))-m(y))+a(y)m(x)m(y)+\chi(y)a(\sigma(y))m(x)m(\sigma(y))|\leq
\delta
 \end{equation} for all $x,y\in G$ and for some $\delta\geq0.$ By
 replacing $y$ by $\sigma(y)$ in (\ref{eq315}) and using
 $\chi(y\sigma(y))=1$ we get
 \begin{equation}\label{eq316}
   |a(x)m(x)(m(y)-\chi(y)m(\sigma(y)))+\chi(y)a(\sigma(y))m(x)m(\sigma(y))+a(y)m(x)m(y)|\leq
\delta.
 \end{equation}
 Subtracting (\ref{eq315}) from  (\ref{eq316}) we get after some
simplifications that
\begin{equation}\label{eq317}
   |2a(x)m(x)||(m(y)-\chi(y)m(\sigma(y)))|\leq
2\delta
 \end{equation} for all $x,y\in G.$ Since $|m(x)|=1$ and $a$ is
 unbounded then we get $m(y)=\chi(y)m(\sigma(y)))$ for all $y\in
 G.$\\
 Now, we will show that $l=ag$ satisfies
 $l(xy)+\chi(y)l(\sigma(y)x)=2l(x)m(y)$. For all $x,y\in G$ we have
 $$l(xy)+\chi(y)l(\sigma(y)x)-2l(x)m(y)$$$$=(a(x)+a(y))m(x)m(y)+\chi(y)(a(\sigma(y))+a(x))m(\sigma(y))m(x)-2a(x)m(x)m(y)$$
 $$=(a(y)+a(\sigma(y)))m(x)m(y).$$
 Since
 \begin{equation}\label{eq318}
 |h(xy)+\chi(y)h(\sigma(y))-2h(x)g(y)|\leq \beta \end{equation} for some $\beta\geq
 $, $h(e)=0$, $h=m(a+b)$,
 $m$ is multiplicative, $|m(x)|=1$, $m=\chi m\circ\sigma$  and $b$ is bounded then
 if we put $x=e$ in (\ref{eq318}) we get
 $$|m(y)a(y)+m(y)b(y)+\chi(y)m(\sigma(y))a(\sigma(y))+\chi(y)m(\sigma(y))b(\sigma(y))|\leq\beta.$$
 This means that the function  $y\longrightarrow
 |a(y)+a(\sigma(y))|$ is bounded. Since $a+a\circ\sigma$ is an
 additive map then we get $a(y)+a(\sigma(y)=0$ for all $y\in G$ and
 we conclude that $l(xy)+\chi(y)l(\sigma(y)x)=2l(x)g(y)$ for all $x,y\in
 G,$ and we see that we deal with case (iii).\\
 If $f,g$ are unbounded, then by using (\ref{eq401}) we get that either $f_a=0$
 for all $a\in G$ or $f_a$ is unbounded for all $a\in  G$. Indeed, if there exists $a\in G$ with $f_a\neq0$ and $f_a$ bounded, so  from inequality (\ref{eq401})
  with $x=x_0$ where $f_a(x_0)\neq0$ we get $g$ bounded which contredicts the assumption that $g$ is unbounded.\\ In this case  we have
 the following possibilities:\\If $f_a=0$ for all $a\in G$ then
 $f(xy)=f(x)g(y)$ for all $x,y\in G$ and this implies that $g$ is
 multiplicative and $f=f(e)g$. Substituting this into (\ref{eq313})
 we get after reduction that
 $|g(x)|\chi(y)g(\sigma(y))-g(y)|\leq\gamma$
  for all $x,y\in
 G$ and for some $\gamma\geq0$. Since $g$ is unbounded we deduce
 that $g=\chi g\circ\sigma$. So,  $g$ satisfies equation (\ref{eq17}),
 and
the pair $f,g$
 satisfies equation (\ref{eq16}). We deal with case (iv)(1).\\
 If there exists $a\in G$ such that $f_a\neq 0$, then by using the
 above notice we get $f_a$ is unbounded for all $a\in G$. For the
 rest of the proof we put $a=e$ and we will discuss two cases. \\
 First Case: If $f_e, g$ are linearly dependent modulo the spaces of
 complex bounded function on $G$ (see \cite{z4}), then there exists a constant  $\lambda\in\mathbb{C}^{\ast}$ and a bounded function $b$ on $G$ such that
 $g=\frac{1}{2\lambda}f_e+b$.  Substituting this into inequality
 (\ref{eq401}) we get
 $$|f_e(xy)-f_e(x)[\frac{1}{2\lambda}f_e(y)+b(y)]-f_e(y)[\frac{1}{2\lambda}f_e(x)+b(x)]|\leq
|\frac{1}{2\lambda}f_e(x)+b(x)]|\delta+\frac{3}{2}\delta$$ \text{for
all }\;$x,y\in G$, so we have
$$
|f_e(xy)-(\frac{1}{\lambda}f_e(x)+b(x))f_e(y)|\leq |f_e(x)||b(y)|+
|\frac{1}{2\lambda}f_e(x)+b(x)]|\delta+\frac{3}{2}\delta.$$ Thus the
function $y\longrightarrow
f_e(xy)-(\frac{1}{\lambda}f_e(x)+b(x))f_e(y)$ is bounded for all
$x\in G.$ Since $f_e$ is unbounded then from Theorem 3.4 (with  $V$
is the space of bounded function on $G$) we get
$m=\frac{1}{\lambda}f_e+b$  multiplicative, $f_e=\lambda m-\lambda
b$, $g=\frac{m}{2}+\frac{b}{2}$ and
$f=f_e+f(e)g=(\lambda+\frac{f(e)}{2})m+(\frac{f(e)}{2}-\lambda)b=\alpha
m+\beta b$ Substituting this into bounded function
$B(x,y)=f(xy)+\chi(y)f(\sigma(y)x)-2f(x)g(y)$
  we find after reduction that
 \begin{equation}\label{eq320}
 \alpha m(x)[\chi(y)m(\sigma(y))-b(y)]=\beta b(x)m(y)+\beta
  b(x)b(y)-\beta \chi(y)b(\sigma(y)x)+B(x,y)\end{equation}
 for all $x,y\in G.$ Since $b$ is bounded, $m$ is
  unbounded and $|\chi(y)|=1$ then there exists $a\in G$ such that
  $\chi(a)m(\sigma(a))-b(a)\neq 0$. From (\ref{eq320}) we conclude
  that $m$ is a bounded multiplicative and this case does not apply,
  because $m$ is unbounded. So we have the second case:\\
  Case 2: $f_e, g$ are linearly independent modulo the spaces
  of complex bounded function on $G$. From inequality (\ref{eq401})
  and Theorem 3.5, (with  $V$ is the space of bounded function on $G$) reveals that the pair
  $(f_e,g)$ is a solution of the sine addition formulas
  \begin{equation}\label{eq321}
    f_e(xy)=f_e(x)g(y)+f_e(y)g(x)
  \end{equation}for all $x,y\in G$, so we known from [\cite{st3}, Corollary
  4.4] that there are only the following  possibilities:\\
  (1) $f_e=cm$ and $g=\frac{m}{2}$ for some multiplicative function  $m:$ $G\longrightarrow
  \mathbb{C}$. Here $f=f_e+f(e)g=(c+\frac{f_e}{2})m=\gamma m$. Substituting this into bounded function
$B(x,y)=f(xy)+\chi(y)f(\sigma(y)x)-2f(x)g(y)$   we find after
reduction that $\beta\chi(y)m(\sigma(y))m(x)=B(x,y)$. This means
that $m$ is a bounded multiplicative and this case does not apply,
because $m$ is unbounded.\\
(2) $g=m$ and $f_e=am$ for some multiplicative function $m:$
$G\longrightarrow  \mathbb{C}$ and $a:$ $G\longrightarrow
\mathbb{C}$ an additive map. In this case $f=(a+f(e))m$, so we find
after reduction that the bounded function:
$$f(xy)+\chi(y)f(\sigma(y)x)-2f(x)g(y)=(a(x)+a(y)+f(e))m(x)m(y)$$$$+\chi(y)[(a(\sigma(y))+a(x)+f(e))m(\sigma(y))m(x)-2(a(x)+f(e))m(x)m(y)]$$
\begin{equation}\label{eq322}
=(a(x)+f(e))m(x)(\chi(y)m(\sigma(y))-m(y))+m(x)[\chi(y)a(\sigma(y))m(\sigma(y))+a(y)m(y)].\end{equation}
If we replace $y$ with $\sigma(y)$  in (\ref{eq322}), and after we
multiply equation obtained by $\chi(y)$ and using
$\chi(y\sigma(y))=1$ we get
\begin{equation}\label{eq323}
(a(x)+f(e))m(x)(m(y)-\chi(y)m(\sigma(y)))+m(x)(a(y)m(y)+\chi(y)a(\sigma(y))m(\sigma(y))\end{equation}
which is also a bounded function. Subtracting (\ref{eq322}) from
(\ref{eq323}) we get after some simplifications that the function
$(x,y)\longmapsto m(x)(a(x)+f(e))(m(y)-\chi(y)m(\sigma(y)))$ is
bounded. Since $f=m(a+f(e))$ is unbounded then we get $m=\chi
m\circ\sigma.$ Now, we will verify that the pair $(f,g)$ is a
solution of equation (\ref{eq16}). for all $x,y\in G$ we have
 $$f(xy)+\chi(y)f(\sigma(y)x)-f(x)g(y)$$$$=(a(x)+a(y)+f(e))m(x)m(y)$$$$+\chi(y)[a(\sigma(y))+a(x)+f(e))m(\sigma(y))m(x)]-2(a(x)+f(e))m(x)m(y)$$
 $$=(a(y)+a(\sigma(y)))m(x)m(y).$$
 Since $(x,y)\longmapsto f(xy)+\chi(y)f(\sigma(y)x)-f(x)g(y)$ is a
 bounded function, then we have
 $(x,y)\longmapsto(a(y)+a(\sigma(y)))m(x)m(y)$ is also
 bounded. Since $m$ is unbounded then we get the desired result, so
 we see that we deal with case (iv) (2).\\
 (3) There exit two different characters $m,M$ and a constant $c\in
 \mathbb{C}^{\ast}$ such that $g=\frac{m+M}{2}$ and $f_e=c(m-M)$. In this
 case $f=f_e+f(e)g=(c+\frac{f(e)}{2})m-(c-\frac{f(e)}{2})m=\alpha m-\beta
 M$, where $\alpha=c+\frac{f(e)}{2}$ and $\beta=c-\frac{f(e)}{2}$.
  Substituting this into bounded function
$B(x,y)=f(xy)+\chi(y)f(\sigma(y)x)-2f(x)g(y)$   we find after
reduction that
\begin{equation}\label{eq324}
\alpha m(x)(\chi(y)m(\sigma(y))-M(y))+\beta
M(x)(m(y)-\chi(y)M(\sigma(y)))=B(x,y).\end{equation} If we replace
$y$ with $\sigma(y)$  in (\ref{eq324}), and after we multiply
equation obtained by $\chi(y)$ and using $\chi(y\sigma(y))=1$ we get
\begin{equation}\label{eq325}
\alpha m(x)(m(y))-\chi(y)M(\sigma(y)))+\beta
M(x)(\chi(y)m(\sigma(y))-M(y))=\chi(y)B(x,\sigma(y)).\end{equation}
If we add   (\ref{eq324}) to (\ref{eq325}) we get after some
simplifications that the function $(x,y)\longmapsto (\alpha
m(x)+\beta
M(x))[(m(y)-\chi(y)M(\sigma(y)))+(\chi(y)m(\sigma(y))-M(y))]$ is
bounded. Since $\alpha m+\beta M=2c g+\frac{f(e)}{c}f_e$ and $g,f_e$
are linearly independent modulo the space of complex bounded
functions on $G$, $\alpha m+\beta M=2c g+\frac{f(e)}{c}f_e$ is
unbounded then we get $m-\chi M\circ\sigma=M-\chi m\circ\sigma$.
Now,  the bounded function (\ref{eq324}) can be written as follows
\begin{equation}\label{eq326}
f(xy)+\chi(y)f(\sigma(y)x)-2f(x)g(y)=(\alpha m(x)-\beta
M(x))(\chi(y)m(\sigma(y))-M(y))\end{equation}$$=f(x)(\chi(y)m(\sigma(y))-M(y))$$
Since $f$ is assumed to be unbounded then we get
$\chi(y)m(\sigma(y))=M(y)$ for all $y\in G$ and $g$ take the
expression: $g=\frac{m+\chi m\circ\sigma}{2}$.  Equation
(\ref{eq326}) show that the pair $(f,g)$ satisfies equation
(\ref{eq16}). We see that we deal with case (iv) (3) and this
completes the proof.\end{proof} As an application we get the
superstability of the functional equation (\ref{eq35}).
\begin{cor}\label{eq17} Let $\delta\geq0$. Let $G$ be a group with
 identity element,  $\sigma$: $G\longrightarrow G$ an involutive
 homomorphism
 and $\chi$: $G\longrightarrow \mathbb{C}$ be a unitary character of $G$
such that $\chi(x\sigma(x))=1$ for all $x\in G$. Let $f$:
$G\rightarrow \mathbb{C}$ such that
 \begin{equation}\label{eq01}
     |f(xy)+\chi(y)f(\sigma(y)x)-2f(x)f(y)| \leq\delta\end{equation}
     for all $x,y\in G$. Then either $f$ is bounded or $f$ satisfies
     equation (\ref{eq35})
\end{cor}
In \cite{el4}, the authors presented some rich ideas on the study of
the superstability of symmetrized multiplicative Cauchy   equation
\begin{equation}\label{eq600}
f (xy) + f (yx) = 2f (x)f(y)\;  x, y\in G.\end{equation} However, we
have  formulate the problem as an  open problem. The solutions of
equation (\ref{eq600}) are multiplicative functions (see for exapmle
\cite{st40}). In the following, we give the affirmative answer. If
we put $\chi=1$ and $\sigma=I$ in Corollary 3.8, where $I$ denotes
the identity map we get
\begin{cor}\label{eq17} Let $\delta\geq0$. Let $G$ be a group with  identity
element. Let $f :G\rightarrow \mathbb{C}$ such that
 \begin{equation}\label{eq01}
     |f(xy)+f(yx)-2f(x)f(y)| \leq\delta\end{equation}
     for all $x,y\in G$. Then either $f$ is bounded or $f$ is
     multiplicative.\end{cor}

 Elqorachi Elhoucien, Department of Mathematics,
Faculty of Sciences, University  Ibn Zohr, Agadir, Morocco, \\
E-mail: elqorachi@hotmail.com\\
 Redouani Ahmed, Department of Mathematics,
Faculty of Sciences, University  Ibn Zohr, Agadir, Morocco, \\
E-mail: Redouani-ahmed@yahoo.fr
\end{document}